\documentclass[11pt,reqno]{amsart}
\usepackage{graphicx}
\usepackage{tikz}
\usepackage{soul}
\usetikzlibrary{shapes.geometric}
\usepackage{amssymb, mathtools}
\usepackage{array}
\usepackage{comment}
\usepackage{enumerate}
\usepackage{caption}
\usepackage{hyperref}
\usepackage{cleveref}
\usepackage[a4paper,margin=3cm]{geometry}
\hypersetup{colorlinks,linkcolor=blue ,citecolor=blue, urlcolor=blue}

\numberwithin{equation}{section}  

\usepackage[backend=biber]{biblatex}
\theoremstyle{definition}

\theoremstyle{plain}
\newtheorem{theorem}{Theorem}[section]
\newtheorem{lemma}[theorem]{Lemma}
\newtheorem{proposition}[theorem]{Proposition}
\newtheorem{corollary}[theorem]{Corollary}

\DeclareMathOperator{\conv}{conv}
\newcommand{\simp}{\mathrm{simp}}
\newcommand{\op}{\mathrm{op}}
\newcommand{\vol}{\mathrm{vol}}

\newcommand{\filter}[1]{\langle #1 \rangle_\uparrow}

\begin{document}
\title{Simplex inequalities of order and chain polytopes of recursively defined posets}
\author{Ragnar Freij-Hollanti, Teemu Lundström}

\address{Ragnar Freij-Hollanti, Department of Mathematics and Systems Analysis, Aalto University, Espoo, Finland
}
\email{ragnar.freij@aalto.fi}
\address{Teemu Lundström, Department of Mathematics and Systems Analysis, Aalto University, Espoo, Finland
}
\email{teemu.lundstrom@aalto.fi}


\begin{abstract}
 In this paper, we study the simplex faces of the order polytope $\mathcal{O}(P)$ and the chain polytope $\mathcal{C}(P)$ of a finite poset $P$.  We show that, if $P$ can be recursively constructed from $\mathbf{X}$-free posets using disjoint unions and ordinal sums, then $\mathcal{C}(P)$ has at least as many $k$-dimensional simplex faces as $\mathcal{O}(P)$ does, for each dimension $k$. This generalizes a previous result of Mori, both in terms of the dimensions of the simplices and in terms of the class of posets considered.
\end{abstract}
\maketitle

\section{Introduction}
The order polytope $\mathcal{O}(P)$ and chain polytope $\mathcal{C}(P)$  are two important geometric invariants associated to a finite poset $P$ on $n$ elements. They were introduced in 1986 by
Stanley~\cite{Stanley}, and share many important geometric features. In particular, they are both $0/1$-polytopes in $\mathbb{R}^P$, they both have the same dimension $\dim(\mathcal{O}(P))=\dim(\mathcal{C}(P))=|P|$, volume $\vol(\mathcal{O}(P))=\vol(\mathcal{C}(P))=\frac{1}{n!}|\{\textrm{linear extensions of } P\}|$. They also have the same number of edges, $f_1(\mathcal{O}(P))=f_1(\mathcal{C}(P))=|P|$, although this number does not have quite as nice an interpretation as $f_0$ and $\dim$~\cite{edges}.
Moreover, their toric rings are both examples
of algebras with straightening laws on distributive lattices~\cite{Hibi87,Hibi&Li2}. 

It was proven in~\cite{Hibi&Li} that $\mathcal{O}(P)$ and $\mathcal{C}(P)$ are unimodularly equivalent if and only if $P$ does not contain a copy of the poset $\mathbf{X}$ in~\Cref{fig:X} as a subposet. We say that such a poset is $\mathbf{X}$-free. It was also shown in~\cite{Hibi&Li} that the number of facets of the two polytopes satisfies $f_{n-1}(\mathcal{O}(P))\leq f_{n-1}(\mathcal{C}(P))$, and that equality holds if and only if $P$ is $\mathbf{X}$-free. In the same paper it was conjectured that the inequality $f_{k}(\mathcal{O}(P))\leq f_{k}(\mathcal{C}(P))$ holds for all values of $k=0, \dots, n-1$. 
This has later been known as the Hibi-Li conjecture.

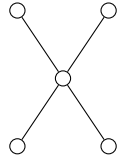
\begin{figure}[htb]
    \centering
    \begin{tikzpicture}
    [scale=0.6,
    dot/.style={draw, circle, minimum size=2mm, inner sep=0pt, fill=white},
    ]
\node(a)[dot] at (0,0){};
\node(b)[dot] at (2,0){};
\node(c)[dot] at (1,1.5){};
\node(d)[dot] at (0,3){};
\node(e)[dot] at (2,3){};

\draw (a) -- (c);
\draw (b) -- (c);
\draw (d) -- (c);
\draw (e) -- (c);
\end{tikzpicture}
    \caption{The smallest poset $\mathbf{X}$ for which $\mathcal{O}(\mathbf{X})$ and $\mathcal{C}(\mathbf{X})$ are not unimodularly equivalent.}\label{fig:X}
\end{figure}

The Hibi-Li conjecture was proven for so-called maximal ranked posets in~\cite{Ahmad-Fourier-Joswig}, and for a more general class $\mathcal{F}$ of posets by the authors in~\cite{f-vector_inequalities}. The class $\mathcal{F}$ includes both $\mathbf{X}$-free posets and series-parallel posets~\cite{Jung}, and in particular all maximal ranked posets. 

The Hibi-Li conjecture also inspired research into the enumeration of faces of given combinatorial type in the poset polytopes, and especially into combinatorial descriptions of the simplex faces in $\mathcal{O}(P)$ and $\mathcal{C}(P)$~\cite{mori_simplex_faces}. Using this characterization, Mori proved that if $P$ is maximal ranked, then $\mathcal{C}(P)$ has at least as many triangular faces as $\mathcal{O}(P)$ does, and that equality holds if and only if $P$ is $\mathbf{X}$-free~\cite{mori_maximal_ranked}. This result was later generalized by Mori together with the authors to hold for arbitrary posets $P$~\cite{two_dim_faces}.

In this paper, we generalize the first half of Mori's result to hold for the number of simplices of all dimensions, and to the same class $\mathcal{F}$ of posets as were studied in~\cite{f-vector_inequalities}, vastly generalizing maximal ranked posets. 
The proofs use very similar techniques as were used in \cite{f-vector_inequalities}. 
Characterizing posets for which equality holds between the number of simplices in any given dimension is left as a future research direction.

\section{Preliminaries}
In this section we introduce some notation and definitions used throughout the paper.
For undefined polytope terminology, see e.g. \cite{Ziegler} or \cite[Chapter 15]{handbook}.
For undefined poset terminology see e.g.\ \cite{Stanley_EC1}.
Our definitions and notations mostly follow those in our earlier paper \cite{f-vector_inequalities}.

All posets in this paper are assumed to be finite.
A non-empty poset $P$ is said to be \emph{connected} if for all $x,y \in P$ there exists $p_0,\dots,p_k \in P$ such that  $x = p_0 \perp x_1 \perp \cdots \perp p_k = y$
where $p_i \perp p_j$ means that $p_i$  and $p_j$ are comparable.
Given an element $p \in P$ we let $\filter{p} \coloneqq \{ q \in P \colon q \ge p \}$ be the principal filter generated by $p$.
Given posets $A$ and $B$, their disjoint union is denoted by $A \sqcup B$ and their ordinal sum is denoted by $A < B$.
The opposite or dual poset of $P$ is denoted by $P^\op$.

In any polytope, we count the empty set and the full polytope to be faces where the empty set has dimension $-1$.
We also consider the empty set to be a simplex.
The \emph{join} of polytopes $\mathcal{P}$ and $\mathcal{Q}$ is denoted by $\mathcal{P} * \mathcal{Q}$ and we consider this construction only up to combinatorial isomorphism.

Let $\mathcal{P}$ be a polytope.
We let $f_k(\mathcal{P})$ denote the number of $k$-faces of $\mathcal{P}$ and $s_k(\mathcal{P})$ denote the number of $k$-simplices in $\mathcal{P}$, by which we mean $k$-dimensional simplex faces of $\mathcal{P}$.
We let $L(\mathcal{P})$ denote the face lattice of $\mathcal{P}$ and we let $L_\simp(\mathcal{P})$ denote the set of simplex faces of $\mathcal{P}$.
Analogously to the $f$-polynomial, we define 
\begin{equation*}
    S_\mathcal{P}(x) \coloneqq \sum_{F \in L_\simp(\mathcal{P})} x^{\dim(F) + 1} = \sum_{i=-1}^{\dim(\mathcal{P})} s_i(\mathcal{P}) x^{i+1} \in \mathbb{Z}_{\ge 0}[x].
\end{equation*}

In this paper we study ordinal sums of posets which correspond to subdirect sums of order and chain polytopes (\Cref{ordinal_sum}).
In these constructions the faces come in two types: those that contain the origin and those that do not (\Cref{faces_of_PvQ}).
We therefore introduce the following additional notation.

Let $\mathcal{P}$ be a polytope containing the origin as a vertex.
We let  $s_k^0(\mathcal{P})$ denote the number of $k$-simplices in $\mathcal{P}$ that contain the origin and $s_k^1(\mathcal{P})$ the number of $k$-simplices in $\mathcal{P}$ that do not contain the origin.
We let $L^0(\mathcal{P})$ and $L^1(\mathcal{P})$ denote the sets of faces containing and not containing the origin respectively.
The restrictions of these to simplex faces are denoted by $L_\simp^0(\mathcal{P})$ and $L_\simp^1(\mathcal{P})$.
We also define the polynomials
\begin{align*}
    S_\mathcal{P}^0(x) &\coloneqq \sum_{F \in L^0_\simp(\mathcal{P})} x^{\dim(F) + 1} = \sum_{i=0}^{\dim(\mathcal{P})} s_i^0(\mathcal{P})x^{i+1}  \in \mathbb{Z}_{\ge 0}[x] \\
    S_\mathcal{P}^1(x) &\coloneqq \sum_{F \in L^1_\simp(\mathcal{P})} x^{\dim(F) + 1} = \sum_{i=-1}^{\dim(\mathcal{P})} s_i^1(\mathcal{P})x^{i+1} \in \mathbb{Z}_{\ge 0}[x].
\end{align*}
Thus $S_\mathcal{P}(x) = S_\mathcal{P}^0(x) + S_\mathcal{P}^1(x)$.

Given two polynomials $f(x),g(x) \in \mathbb{Z}_{\ge 0}[x]$ we write $f(x) \le g(x)$ if the inequality holds coefficient-wise.
Hence, for polytopes $\mathcal{P}$ and $\mathcal{Q}$ the inequality $S_\mathcal{P}(x) \le S_\mathcal{Q}(x)$ holds if and only if $s_k(\mathcal{P}) \le s_k(\mathcal{Q})$ for all $k$.

Next, we quickly give the definitions and basic properties of order and chain polytopes.
For more on these polytopes, see for example Stanley's original paper \cite{Stanley}.

Let $P$ be a poset.
The \emph{order polytope} of $P$ is the polytope
\begin{equation*}
    \mathcal{O}(P) =
    \left\{ x \in \mathbb{R}^P \;\middle|\;
    \begin{gathered} 
        0 \leq x_p \leq 1 \mbox{ for all } p \in P, \\
        x_p \leq x_q \mbox{ if } p \leq q \mbox{ in }P
    \end{gathered}
    \right\}
    \subseteq \mathbb{R}^P.
\end{equation*}
The \emph{chain polytope} of $P$ is the polytope
\begin{equation*}
    \mathcal{C}(P) =
    \left\{ x \in \mathbb{R}^P \;\middle|\;
    \begin{gathered} 
        x_p \geq 0 \mbox{ for all } p \in P,\\
        x_{p_1}+\cdots +x_{p_k} \leq 1 \mbox{ if } p_1 < \cdots < p_k \mbox{ in } P
    \end{gathered}
    \right\}
    \subseteq \mathbb{R}^P.
\end{equation*}
Given a subset $S \subseteq P$ we let $\chi_S \in \mathbb{R}^P$ denote the characteristic vector of $S$.
The vertices of $\mathcal{O}(P)$ are given by $\chi_F$ for all filters $F \subseteq P$, and the vertices of $\mathcal{C}(P)$ are given by $\chi_A$ for all antichains $A \subseteq P$.
The dimension of $\mathcal{O}(P)$ and $\mathcal{C}(P)$ is given by 
\begin{equation*}
    \dim(\mathcal{O}(P)) = \dim(\mathcal{C}(P)) = |P|.
\end{equation*}

In the literature of order polytopes (e.g.\ \cite{cutting,Hibi&Li,mori_simplex_faces}) the definition of $\mathcal{O}(P)$ is sometimes written by requiring that $x_p \le x_q$ for all $p \ge q$ in $P$.
The vertices of the resulting polytope then correspond to order ideals instead of filters.
The polytope defined that way is isomorphic to the polytope $\mathcal{O}(P)$ defined here and one can move bijectively between order ideals and filters by taking complements.

Notice the following edge case.
If $P=\emptyset$ is the empty poset then $\mathcal{O}(P)$ and $\mathcal{C}(P)$ are points, namely the origin in $\mathbb{R}^\emptyset$.
Hence $S^0_{\mathcal{O}(\emptyset)}(x) =S^0_{\mathcal{C}(\emptyset)}(x )= x$ and $S_{\mathcal{O}(\emptyset)}^1(x) = S_{\mathcal{C}(\emptyset)}^1(x) = 0$.

Given polytopes $\mathcal{P} \subseteq \mathbb{R}^m$ and $\mathcal{Q} \subseteq \mathbb{R}^n$ both containing the origin as a vertex, we define their \emph{subdirect sum} as
\begin{equation*}
    \mathcal{P} \vee \mathcal{Q} \coloneqq \conv(\mathcal{P} \times \{ 0 \}^n \cup \{ 0 \}^m \times \mathcal{Q}) \subseteq \mathbb{R}^{m+n}.
\end{equation*}
This is a special case of a construction studied by McMullen in \cite{McMullen} from where we took the name ''subdirect sum''.
The subdirect sums of McMullen interpolates between our notion of subdirect sums and the notion of direct (or free) sums, in which the origin is assumed to be an interior point in both polytopes, rather than a vertex.
For more on these constructions and their terminology, see for example \cite[Section 15.1.3]{handbook}.

The following gives a complete description of the faces of $\mathcal{P} \vee \mathcal{Q}$.
This follows from the more general description of the faces of $\mathcal{P} \vee \mathcal{Q}$ which is briefly mentioned in \cite{McMullen}.
\begin{proposition}[{\cite[Proposition 3.7]{f-vector_inequalities}}]\label{faces_of_PvQ}
    Let $\mathcal{P}$ and $\mathcal{Q}$ be polytopes both containing the origin as a vertex.
    The faces of $\mathcal{P} \vee \mathcal{Q}$ are 
    \begin{enumerate}[(1)]
        \item $F \vee G$ for all faces $F \in L^0(\mathcal{P})$ and $G \in L^0(\mathcal{Q})$, and 
        \item $F * G$ for all faces $F \in L^1(\mathcal{P})$ and $G \in L^1(\mathcal{Q})$.
    \end{enumerate}
\end{proposition}

The ordinal sum of posets relates to subdirect sum in the following way.
This was essentially also noticed in \cite[Lemma 7.2--7.3]{Levelness}.
\begin{proposition}[{\cite[Proposition 4.1--4.2]{f-vector_inequalities}}]\label{ordinal_sum}
    Let $A$ and $B$ be posets.
    Then
    \begin{equation*}
        \mathcal{O}(A<B) \cong \mathcal{O}(A) \vee \mathcal{O}(B^\op)
    \end{equation*}
    and 
    \begin{equation*}
        \mathcal{C}(A<B) = \mathcal{C}(A) \vee \mathcal{C}(B).
    \end{equation*}
\end{proposition}

Hibi and Li gave the following characterization of edges for both polytopes.
Here and throughout this paper $\Delta$ denotes the symmetric difference of sets.
\begin{theorem}[\cite{cutting}]\label{edges}
    Let $P$ be a poset.
    \begin{enumerate}[(1)]
        \item The edges of $\mathcal{O}(P)$ are exactly the sets $\conv(\chi_{F_1},\chi_{F_2})$ where $F_1$ and $F_2$ are filters such that $F_1 \subseteq F_2$ and $F_2 \setminus F_1$ is connected.
        \item The edges of $\mathcal{C}(P)$ are exactly the sets $\conv(\chi_{A_1},\chi_{A_2})$ where $A_1$ and $A_2$ are antichains such that $A_1 \Delta A_2$ is connected.
    \end{enumerate}
\end{theorem}
This characterization has been recently generalized to all simplex faces by Mori. 

\begin{theorem}[\cite{mori_simplex_faces}]\label{mori_theorem}
    Let $P$ be a poset.
    Let $F_1, \dots, F_k$ be distinct filters and let $A_1, \dots, A_k$ be distinct antichains.
    \begin{enumerate}[(1)]
        \item $\conv(\chi_{F_1},\dots,\chi_{F_k})$ is a $(k-1)$-simplex face of $\mathcal{O}(P)$ if and only if $\conv(\chi_{F_i},\chi_{F_j})$ is an edge for all $i \neq j$.
        \item $\conv(\chi_{A_1},\dots,\chi_{A_k})$ is a $(k-1)$-simplex face of $\mathcal{C}(P)$ if and only if $\conv(\chi_{A_i},\chi_{A_j})$ is an edge for all $i \neq j$.
    \end{enumerate}
\end{theorem}

\section{Counting simplex faces}
In this section, we prove theorems related to the number of simplex faces in polytopes, which will allow us to prove our main results in \Cref{section:main}. 

\begin{lemma}\label{vee_simplex}
    Let $\mathcal{P}$ and $\mathcal{Q}$ be polytopes both containing origins as a vertex.
    Then $\mathcal{P} \vee \mathcal{Q}$ is a simplex if and only if both $\mathcal{P}$ and $\mathcal{Q}$ are simplices.
\end{lemma}

\begin{proof}
    In our earlier paper~\cite{f-vector_inequalities} we showed that $\dim(\mathcal{P} \vee \mathcal{Q}) = \dim(\mathcal{P}) + \dim(\mathcal{Q})$ \cite[Proposition 3.1]{f-vector_inequalities}, and that $f_0(\mathcal{P} \vee \mathcal{Q}) = f_0(\mathcal{P}) + f_0(\mathcal{Q}) - 1$ \cite[Proposition 3.2]{f-vector_inequalities}.
    
    If both $\mathcal{P}$ and $\mathcal{Q}$ are simplices, then $f_0(\mathcal{P}) = \dim(\mathcal{P}) + 1$  and $f_0(\mathcal{Q}) = \dim(\mathcal{Q}) + 1$.
    Thus $f_0(\mathcal{P} \vee \mathcal{Q}) = \dim(\mathcal{P} \vee \mathcal{Q}) +1$ and therefore $\mathcal{P} \vee \mathcal{Q}$ is a simplex.
    
    If, say, $\mathcal{P}$ is not a simplex, then $f_0(\mathcal{P}) > \dim(\mathcal{P}) + 1$ and $f_0(\mathcal{Q}) \ge \dim(\mathcal{Q}) + 1$.
    It follows that $f_0(\mathcal{P} \vee \mathcal{Q}) > \dim(\mathcal{P} \vee \mathcal{Q}) + 1$, implying that $\mathcal{P} \vee \mathcal{Q}$ is not a simplex.
\end{proof}

For any polytopes $\mathcal{P}$ and $\mathcal{Q}$ we have $f_0(\mathcal{P} * \mathcal{Q}) = f_0(\mathcal{P}) + f_0(\mathcal{Q})$ and $\dim(\mathcal{P} * \mathcal{Q}) = \dim(\mathcal{P}) + \dim(\mathcal{Q}) + 1$.
Thus by a similar vertex--dimension counting as in \Cref{vee_simplex} we obtain

\begin{lemma}\label{join_simplex}
    Let $\mathcal{P}$ and $\mathcal{Q}$ be any polytopes.
    Then $\mathcal{P} * \mathcal{Q}$ is a simplex if and only if both $\mathcal{P}$ and $\mathcal{Q}$ are simplices.
\end{lemma}

\begin{proposition}\label{s_poly_vee}
    Let $\mathcal{P}$ and $\mathcal{Q}$ be polytopes both containing origins as a vertex.
    Then
    \begin{equation*}
        S_{\mathcal{P} \vee \mathcal{Q}}(x) = \frac{1}{x} S^0_\mathcal{P}(x) S^0_{\mathcal{Q}}(x) + S_{\mathcal{P}}^1(x) S^1_{\mathcal{Q}}(x).
    \end{equation*}
\end{proposition}

\begin{proof}
    In our earlier paper \cite[Proposition 3.7]{f-vector_inequalities} we showed that 
    \begin{align*}
        L^0(\mathcal{P}) \times L^0(\mathcal{Q}) &\longrightarrow L^0(\mathcal{P} \vee \mathcal{Q}) \\
        (F,G) &\longmapsto F \vee G
    \end{align*}
    and
    \begin{align*}
        L^1(\mathcal{P}) \times L^1(\mathcal{Q}) &\longrightarrow L^1(\mathcal{P} \vee \mathcal{Q}) \\
        (F,G) &\longmapsto F * G
    \end{align*}
    are a bijections.
    By \Cref{join_simplex} we can restrict both maps to simplex faces to obtain bijections
    \begin{align*}
        L^0_\simp(\mathcal{P}) \times L^0_\simp(\mathcal{Q}) &\longrightarrow L^0_\simp(\mathcal{P} \vee \mathcal{Q}) \\
        (F,G) &\longmapsto F \vee G
    \end{align*}
    and
    \begin{align*}
        L^1_\simp(\mathcal{P}) \times L^1_\simp(\mathcal{Q}) &\longrightarrow L^1_\simp(\mathcal{P} \vee \mathcal{Q}) \\
        (F,G) &\longmapsto F* G.
    \end{align*}
    
    We may therefore compute
    \begin{align*}
        S_{\mathcal{P} \vee \mathcal{Q}}(x) &= S^0_{\mathcal{P} \vee \mathcal{Q}}(x) + S^1_{\mathcal{P} \vee \mathcal{Q}}(x) \\
        &= \sum_{H \in L^0_\simp(\mathcal{P} \vee \mathcal{Q})} x^{\dim(H) + 1} + \sum_{H \in L^1_\simp(\mathcal{P} \vee \mathcal{Q})} x^{\dim(H) + 1} \\
        &= \sum_{\substack{F \in L^0_\simp(\mathcal{P}) \\ G \in L^0_\simp(\mathcal{Q})}} x^{\dim(F \vee G) + 1} + \sum_{\substack{F \in L^1_\simp(\mathcal{P}) \\ G \in L^1_\simp(\mathcal{Q})}} x^{\dim(F * G) + 1}  \\
        &= \frac{1}{x} \sum_{\substack{F \in L^0_\simp(\mathcal{P}) \\ G \in L^0_\simp(\mathcal{Q})}} x^{\dim(F) + 1} x^{\dim(G) + 1}  + \sum_{\substack{F \in L^1_\simp(\mathcal{P}) \\ G \in L^1_\simp(\mathcal{Q})}} x^{\dim(F) + 1} x^{\dim(G) + 1}  \\
        &= \frac{1}{x} S^0_{\mathcal{P}}(x) S^0_{\mathcal{Q}}(x) + S^1_\mathcal{P}(x)S^1_\mathcal{Q}(x).
    \end{align*}
\end{proof}

Combining \Cref{ordinal_sum} and \Cref{s_poly_vee} we obtain

\begin{corollary}\label{S_poly_order_chain}
    For any posets $A$ and $B$,
    \begin{align*}
        S_{\mathcal{O}(A<B)}(x) &= \frac{1}{x} S^0_{\mathcal{O}(A)}(x)S^0_{\mathcal{O}(B^\op)}(x) + S^1_{\mathcal{O}(A)}(x)S^1_{\mathcal{O}(B^\op)}(x) \\
        S_{\mathcal{C}(A<B)}(x) &= \frac{1}{x} S^0_{\mathcal{C}(A)}(x)S^0_{\mathcal{C}(B)}(x) + S^1_{\mathcal{C}(A)}(x)S^1_{\mathcal{C}(B)}(x).
    \end{align*}
\end{corollary}

\begin{proposition}\label{cp_le_op}
    For any poset $P$,
    \begin{equation*}
        S^0_{\mathcal{C}(P)} (x) \le S_{\mathcal{O}(P)}^0(x).
    \end{equation*}
\end{proposition}

\begin{proof}
    Fix an integer $k \in \{ 0,1,\dots,|P| \}$.
    We need to show $s^0_k(\mathcal{C}(P)) \le s^0_k(\mathcal{O}(P))$.
    Let
    \begin{equation*}
        \conv(\chi_{\emptyset},\chi_{A_1},\dots,\chi_{A_k})
    \end{equation*}
    be a $k$-simplex  face in $\mathcal{C}(P)$ containing the origin where $A_1,\dots,A_k$ are antichains of $P$.
    By \Cref{edges} we know that $A_i \Delta \emptyset$ and $A_i \Delta A_j$ are connected for all $i,j \in \{ 1,\dots,k \}$.
    Thus $A_i = \{ p_i \}$ for some $p_i \in P$ for all $i \in \{ 1,\dots,k \}$.
    Since $A_i \Delta A_j = \{ p_i,p_j \}$ is connected, the elements $p_i$ and $p_j$ need to be comparable.
    Therefore $\{ p_1,\dots,p_k \}$ is a chain in $P$.
    Without loss of generality, suppose $p_1 < \dots < p_k$.
    We map this simplex to
    \begin{equation*}
        \conv(\chi_\emptyset,\chi_{\filter{p_1}},\dots,\chi_{\filter{p_k}}) \subseteq \mathcal{O}(P).
    \end{equation*}
    Here $\emptyset \subsetneq \filter{p_k} \subsetneq \filter{p_{k-1}} \subsetneq \cdots \subsetneq \filter{p_1}$.
    Clearly $\filter{p_i} = \filter{p_i} \setminus \emptyset$ is connected for all $i \in \{ 1,\dots,k \}$.
    Furthermore, for all $i,j \in \{ 1,\dots,k \}$ with $i<j$ the difference $\filter{p_i} \setminus \filter{p_j}$ is connected since every element in it is comparable to $p_i$.
    Hence by \Cref{edges} the vertices $\chi_\emptyset,\chi_{\filter{p_1}},\dots,\chi_{\filter{p_k}}$ are pairwise adjacent in the 1-skeleton of $\mathcal{O}(P)$.
    Thus by \Cref{mori_theorem} we conclude that $\conv(\chi_\emptyset,\chi_{\filter{p_1}},\dots,\chi_{\filter{p_k}})$ is a $k$-simplex face in $\mathcal{O}(P)$.
    
    This gives us a map
    \begin{align*}
        \{ F \in L^0_\simp(\mathcal{C}(P)) \colon \dim(F) = k \} &\longrightarrow \{ F \in L_\simp^0(\mathcal{O}(P)) \colon \dim(F) = k \} \\
        \conv(\chi_\emptyset,\chi_{A_1},\dots,\chi_{A_k}) &\longmapsto \conv(\chi_\emptyset,\chi_{\filter{p_1}},\dots,\chi_{\filter{p_k}})
    \end{align*}
    and it is straightforward to check that this map is injective.
    Hence $s_k^0(\mathcal{C}(P)) \le s_k^0(\mathcal{O}(P))$ as desired.
\end{proof}

\begin{proposition}\label{L0_vs_L1}
    Let $\mathcal{P}$ be a polytope containing the origin as a vertex.
    Then for all $k \ge 0$,
    \begin{equation*}
        s_k^0(\mathcal{P}) \le s_{k-1}^1(\mathcal{P}).
    \end{equation*}
\end{proposition}

\begin{proof}
    Given a $k$-simplex in $\mathcal{P}$ containing the origin we may map it to a $(k-1)$-simplex in $\mathcal{P}$ not containing the origin by
    \begin{equation*}
        \conv(0,v_1,\dots,v_k) \longmapsto \conv(v_1,\dots,v_k).
    \end{equation*}
    The injectivity of this map proves the proposition.
\end{proof}

\begin{corollary}\label{L0_poly_vs_L1_poly}
    For any polytope $\mathcal{P}$ containing the origin as a vertex,
    \begin{equation*}
        S_{\mathcal{P}}^0(x) \le x S_{\mathcal{P}}^1(x).
    \end{equation*}
\end{corollary}

\begin{proof}
    Multiplying a polynomial by $x$ shifts the degrees of the terms up by one and both polynomials $S_{\mathcal{P}}^0(x)$ and $xS_\mathcal{P}^1(x)$ have constant term 0.
    The claimed inequality thus follows from \Cref{L0_vs_L1}.
\end{proof}

\begin{lemma}\label{cartesian_prod_simplicies}
    For any polytopes $\mathcal{P}$ and $\mathcal{Q}$,
    \begin{equation*}
        s_0(\mathcal{P} \times \mathcal{Q}) = s_0(\mathcal{P})s_0(\mathcal{Q})
    \end{equation*}
    and for all $k \ge 1$,
    \begin{equation*}
        s_k(\mathcal{P} \times \mathcal{Q}) = s_k(\mathcal{P})s_0(\mathcal{Q}) + s_0(\mathcal{P})s_k(\mathcal{Q}).
    \end{equation*}
\end{lemma}

\begin{proof}
    The first equation is a standard fact since $s_0(-) = f_0(-)$.
    For the second equation, recall that the $k$-faces of $\mathcal{P} \times \mathcal{Q}$ are obtained by taking all cartesian products $F \times G$ where $F \subseteq \mathcal{P}$ and $G \subseteq \mathcal{Q}$ are non-empty faces such that $\dim(F) + \dim(G) = k$.
    Note that the cartesian product of two non-empty polytopes is a simplex if and only if one of the polytopes is a simplex and the other is a point.
    Indeed, if both polytopes contain an edge then their cartesian product will contain a square face, implying that their product is not a simplex.
    With these observations the second equation follows.
\end{proof}

\begin{proposition}\label{cartesian_product_ineq}
    If posets $A$ and $B$ satisfy $S_{\mathcal{O}(A)}(x) \le S_{\mathcal{C}(A)}(x)$ and $S_{\mathcal{O}(B)}(x) \le S_{\mathcal{C}(B)}(x)$, then
    \begin{equation*}
        S_{\mathcal{O}(A \sqcup B)}(x) \le S_{\mathcal{C}(A \sqcup B)}(x).
    \end{equation*}
\end{proposition}

\begin{proof}
    Follows easily from \Cref{cartesian_prod_simplicies} and the fact that $\mathcal{O}(A \sqcup B) = \mathcal{O}(A) \times \mathcal{O}(B)$ and $\mathcal{C}(A \sqcup B) = \mathcal{C}(A) \times \mathcal{C}(B)$.
\end{proof}

Next we define additional polynomials that will help us prove a similar result for ordinal sums.
These polynomials are similar to the polynomials in our earlier paper \cite{f-vector_inequalities}.
For any poset $P$ let
\begin{align*}
    \tilde \alpha_P(x) &\coloneqq \frac{1}{x} S^0_{\mathcal{C}(P)}(x) \\
    \tilde \beta_P(x) &\coloneqq\frac{1}{x}S^0_{\mathcal{O}(P)}(x)  \\
    \tilde \gamma_P(x) &\coloneqq S^1_{\mathcal{O}(P)}(x) \\
    \tilde \delta_P(x) &\coloneqq S^1_{\mathcal{C}(P)}(x).
\end{align*}
Note that $\tilde \alpha_P(x) = \tilde \alpha_{P^\op}(x)$ and $\tilde \delta_P(x) = \tilde \delta_{P^\op}(x)$.

\begin{lemma}\label{abcd}
    For any poset $P$, if $S_{\mathcal{O}(P)}(x) \le S_{\mathcal{C}(P)}(x)$ then
    \begin{enumerate}[(1)]
        \item $0 \le x(\tilde \beta_P(x) - \tilde \alpha_P(x)) \le \tilde \delta_P(x) - \tilde \gamma_P(x)$, and
        \item $0 \le \tilde \alpha_P(x) \le \tilde \beta_P(x) \le \tilde \gamma_P(x) \le \tilde \delta_P(x)$.
    \end{enumerate}
\end{lemma}

\begin{proof}
    For part (1), the assumption $S_{\mathcal{O}(P)}(x) \le S_{\mathcal{C}(P)}(x)$ implies $S_{\mathcal{O}(P)}^0(x) + S_{\mathcal{O}(P)}^1(x) \le S_{\mathcal{C}(P)}^0(x) + S_{\mathcal{C}(P)}^1(x)$ and hence 
    \begin{equation*}
        S^0_{\mathcal{O}(P)}(x) - S^0_{\mathcal{C}(P)}(x) \le S^1_{\mathcal{C}(P)}(x) - S^1_{\mathcal{O}(P)}(x).
    \end{equation*}
    Note that $\tilde \beta_P(x) - \tilde \alpha_P(x) \ge 0$ by \Cref{cp_le_op}.
    From the definitions we then obtain (1).
    
    Let us then prove (2).
    We noticed above that $\tilde \alpha_P(x) \le \tilde \beta_P(x)$.
    The inequality $\tilde \beta_P(x) \le \tilde \gamma_P(x)$ is equivalent with $S_{\mathcal{O}(P)}^0(x) \le x S^1_{\mathcal{O}(P)}(x)$ which holds by \Cref{L0_poly_vs_L1_poly}.
    The inequality $\tilde \gamma_P(x) \le \tilde \delta_P(x)$ follows from part (1).
\end{proof}

\begin{proposition}\label{orderedsum_ineq}
    If $A$ and $B$ are posets that satisfy $S_{\mathcal{O}(A)}(x) \le S_{\mathcal{C}(A)}(x)$ and $S_{\mathcal{O}(B)}(x) \le S_{\mathcal{C}(B)}(x)$ then 
    \begin{equation*}
        S_{\mathcal{O}(A < B)}(x) \le S_{\mathcal{C}(A<B)}(x).
    \end{equation*}
\end{proposition}

This proof is essentially the same as the proof of \cite[Theorem 6.2]{f-vector_inequalities}.
We provide the proof here for completeness.

\begin{proof}
    Since $\mathcal{O}(B^\op) \cong \mathcal{O}(B)$ and $\mathcal{C}(B^\op) = \mathcal{C}(B)$, we have $S_{\mathcal{O}(B^\op)}(x) \le S_{\mathcal{C}(B^\op)}(x)$.
    Thus \Cref{abcd} applies to posets $A,B$ and also $B^\op$.
    By \Cref{S_poly_order_chain} we have 
    \begin{align*}
        S_{\mathcal{O}(A<B)}(x) &= \frac{1}{x} S^0_{\mathcal{O}(A)}(x)S^0_{\mathcal{O}(B^\op)}(x) + S^1_{\mathcal{O}(A)}(x)S^1_{\mathcal{O}(B^\op)}(x) \\
        &= x \tilde \beta_A(x) \tilde \beta_{B^\op}(x) + \tilde \gamma_A(x)  \gamma_{B^\op}(x)
    \end{align*}
    and
    \begin{align*}
        S_{\mathcal{C}(A<B)}(x) &= \frac{1}{x} S^0_{\mathcal{C}(A)}(x)S^0_{\mathcal{C}(B)}(x) + S^1_{\mathcal{C}(A)}(x)S^1_{\mathcal{C}(B)}(x) \\
        &= x \tilde \alpha_A(x) \tilde \alpha_B(x) + \tilde \delta_A(x) \tilde \delta_B(x).
    \end{align*}
    Showing $S_{\mathcal{O}(A<B)}(x) \le S_{\mathcal{C}(A<B)}(x)$ is therefore equivalent to showing
    \begin{equation}
        x \tilde \beta_A(x) \tilde \beta_{B^\op}(x) - x \tilde \alpha_A(x) \tilde \alpha_B(x) \le \tilde \delta_A(x) \tilde \delta_B(x) - \tilde \gamma_A(x) \tilde \gamma_{B^\op}(x). \label{eq:1}
    \end{equation}
    The left-hand side in \eqref{eq:1} can be rewritten as
    \begin{equation}
        x \tilde \alpha_A(x) \Big(\tilde \beta_{B^\op}(x) - \tilde \alpha_B(x) \Big) + x \Big( \tilde \beta_A(x) - \tilde \alpha_A(x)\Big) \tilde \beta_{B^\op}(x) \label{eq:2}
    \end{equation}
    and the right-hand side in \eqref{eq:1} can be rewritten as 
    \begin{equation}
        \tilde \gamma_A(x) \Big( \tilde \delta_B(x) - \tilde \gamma_{B^\op}(x) \Big) + \Big( \tilde \delta_A(x) - \tilde \gamma_A (x)\Big) \tilde \delta_B(x). \label{eq:3}
    \end{equation}
    Our aim is thus to show $\eqref{eq:2} \le \eqref{eq:3}$.
    
    By applying part (1) of \Cref{abcd} to $B^\op$ and part (2) to $A$ we obtain
    \begin{equation*}
        \begin{cases}
            0 \le x (\tilde \beta_{B^\op}(x) - \tilde \alpha_{B^\op}(x)) \le \tilde \delta_{B^\op}(x) - \tilde \gamma_{B^{op}}(x) \\
            0 \le \tilde \alpha_A(x) \le \tilde \gamma_A(x).
        \end{cases}
    \end{equation*}
    Multiplying both sides together here gives us 
    \begin{equation}
        x \tilde \alpha(x) \Big( \tilde \beta_{B^\op}(x) - \tilde \alpha_{B^\op}(x) \Big) \le \tilde \gamma_A(x) \Big( \tilde \delta_{B^\op}(x) - \tilde \gamma_{B^\op}(x) \Big). \label{eq:4}
    \end{equation}
    Applying part (1) of \Cref{abcd} to $A$ and part (2) to $B^\op$ we get
    \begin{equation*}
        \begin{cases}
            0 \le x (\tilde \beta_A(x) - \tilde \alpha_A(x)) \le \tilde \delta_A(x) - \tilde \gamma_A(x) \\
            0 \le \tilde \beta_{B^\op}(x) \le \tilde \delta_{B^\op}(x).
        \end{cases}
    \end{equation*}
    Multiplying both sides together gives us
    \begin{equation}
        x \tilde \beta_{B^\op}(x) \Big( \tilde \beta_A(x) - \tilde \alpha_A(x) \Big) \le \tilde \delta_{B^\op}(x) \Big( \tilde \delta_A(x) - \tilde \gamma_A(x) \Big) \label{eq:5}.
    \end{equation}
    Recall that $\tilde \alpha_{B^\op}(x) = \tilde \alpha_B(x)$ and $\tilde \delta_{B^\op}(x) = \tilde \delta_B(x)$.
    Adding inequalities \eqref{eq:4} and \eqref{eq:5} together thus gives us $\eqref{eq:2} \le \eqref{eq:3}$.
    This finishes the proof.
\end{proof}

\section{Main result and conclusion}\label{section:main}
From  Propositions~\ref{cartesian_product_ineq} and \ref{orderedsum_ineq}, our main result follows immediately. 
\begin{theorem}\label{main_theorem}
    Let $\mathcal{F}$ be the family of posets built by starting with $X$-free posets and using disjoint unions and ordinal sums.
    Then any poset $P \in \mathcal{F}$ satisfies
    \begin{equation*}
        s_k(\mathcal{O}(P)) \le s_k(\mathcal{C}(P))
    \end{equation*}
    for all $k \ge 0$.
\end{theorem}

\begin{proof}
    If $P$ is $X$-free then $\mathcal{O}(P)$ and $\mathcal{C}(P)$ are unimodularly equivalent \cite{Hibi&Li}, so we have an equality $s_k(\mathcal{O}(P)) = s_k(\mathcal{C}(P))$ for all $k$.
    
    Then suppose $A$ and $B$ are posets in $\mathcal{F}$ that satisfy $s_k(\mathcal{O}(A)) \le s_k(\mathcal{C}(A))$ and $s_k(\mathcal{O}(B)) \le s_k(\mathcal{C}(B))$ for all $k$.
    Now for all $k$ we have $s_k(\mathcal{O}(A \sqcup B)) \le s_k(\mathcal{C}(A \sqcup B))$ by \Cref{cartesian_product_ineq} and $s_k(\mathcal{O}(A < B)) \le s_k(\mathcal{C}(A  < B))$ by \Cref{orderedsum_ineq}.
    Hence every poset in the family $\mathcal{F}$ satisfies the claimed inequality.
\end{proof}

\begin{corollary}
    Let $P$ be in the family $\mathcal{F}$ defined in \Cref{main_theorem}.
    Then $S_{\mathcal{O}(P)}(x) = S_{\mathcal{C}(P)}(x)$ if and only if $P$ is $X$-free.
\end{corollary}

\begin{proof}
    If $P$ is $X$-free then $\mathcal{O}(P)$ and $\mathcal{C}(P)$ are unimodularly equivalent, and thus $S_{\mathcal{O}(P)}(x) = S_{\mathcal{C}(P)}(x)$.
    If we have an equality $S_{\mathcal{O}(P)}(x) = S_{\mathcal{C}(P)}(x)$ then in particular $s_2(\mathcal{O}(P)) = s_2(\mathcal{C}(P))$ and thus by \cite[Theorem 4.1]{two_dim_faces} $P$ has to be $X$-free.
\end{proof}

\Cref{main_theorem} generalizes the main result in \cite{mori_maximal_ranked}, where it is shown that $s_2(\mathcal{O}(P)) \le s_2(\mathcal{C}(P))$ holds for all so called maximal ranked posets $P$, which are posets built by taking ordinal sums of antichains.
Such posets are clearly contained in our family $\mathcal{F}$.
On the other hand, in \cite{two_dim_faces} it shown that $s_2(\mathcal{O}(P)) \le s_2(\mathcal{C}(P))$ holds for all posets $P$.
\Cref{main_theorem} generalizes this result to arbitrary simplex faces, but restricting to a smaller class of posets.
In \cite{f-vector_inequalities}, it is proved that $f_k(\mathcal{O}(P)) \le f_k(\mathcal{C}(P))$ holds for any poset $P$ in the family $\mathcal{F}$.
\Cref{main_theorem} proves a variation of this result by restricting to simplex faces.

It seems like a natural question to ask for a formula for the largest dimension of a simplex face 
in $\mathcal{O}(P)$ and $\mathcal{C}(P)$. Since every maximal chain in $P$ gives a simplex in $\mathcal{C}(P)$ by \Cref{mori_theorem}, \Cref{cp_le_op} shows that both polytopes have simplices of dimension equal to the height of $P$. Therefore, the maximal dimensions of simplices are bounded from below by, but not necessarily equal to, the height of the poset. Finding a formula for the largest simplex even for special classes of posets, such as series-parallel posets, appears to be an open problem. Note that, by the results in this paper, the largest dimension of a simplex in $\mathcal{C}(P)$ is at least that of a simplex in $\mathcal{O}(P)$ if $P$ belongs to the recursively defined family $\mathcal{F}$ of posets.

\printbibliography

\end{document}